\documentclass[letterpaper, 10 pt, conference]{ieeeconf}
\IEEEoverridecommandlockouts
\usepackage{cite}
\usepackage{graphicx}      
\usepackage{bbm}
\usepackage[utf8]{inputenc}
\usepackage[T1]{fontenc}

\newcommand{\mS}{\mathcal{S}}

\newcommand{\mP}{\mathcal{P}}

\newcommand{\mE}{\mathcal{E}}
\newcommand{\mF}{\mathcal{F}}

\newcommand{\mD}{\mathcal{D}}
\newcommand{\mG}{\mathcal{G}}

\newcommand{\mL}{\mathcal{L}}

\newcommand{\mV}{\mathcal{V}}

\newtheorem{theorem}{\bf Theorem}

\newtheorem{lemma}{\bf Lemma}

\newtheorem{assumption}{\bf Assumption}

\newtheorem{remark}{\bf Remark}

\newcommand{\vnorm}[1]{\left|\left|#1\right|\right|}

\usepackage{color}
\definecolor{Red}{rgb}{1,0,0}
\definecolor{Blue}{rgb}{0,0,1}
\definecolor{Green}{rgb}{0,1,0}
\definecolor{magenta}{rgb}{1,0,.6}
\definecolor{lightblue}{rgb}{0,.5,1}
\definecolor{lightpurple}{rgb}{.6,.4,1}
\definecolor{gold}{rgb}{.6,.5,0}
\definecolor{orange}{rgb}{1,0.4,0}
\definecolor{hotpink}{rgb}{1,0,0.5}
\definecolor{newcolor2}{rgb}{.5,.3,.5}
\definecolor{newcolor}{rgb}{0,.3,1}
\definecolor{newcolor3}{rgb}{1,0,.35}
\definecolor{darkgreen1}{rgb}{0, .35, 0}
\definecolor{darkgreen}{rgb}{0, .6, 0}
\definecolor{darkred}{rgb}{.75,0,0}

\newcommand{\ts}{t^\star}

\usepackage{amsmath}   

\usepackage{amssymb}
\usepackage{algorithm}
\usepackage{algorithmic}
\usepackage{array}
\usepackage{comment}
\usepackage{url}

\begin{document}

\title{Robust Distributed Averaging on Networks with Adversarial Intervention}

\author{Ali Khanafer, Behrouz Touri, and Tamer Ba\c{s}ar
\thanks{This work was supported in part by an AFOSR MURI Grant FA9550-10-1-0573.}
\thanks{Ali Khanafer and Tamer Ba\c{s}ar are with the Coordinated Science Laboratory, ECE Department, University of Illinois at Urbana-Champaign, USA {\tt\small {khanafe2,basar1}@illinois.edu}}
\thanks{Behrouz Touri is with the ECE Department, Georgia Institute of Technology, USA {\tt\small touri@gatech.edu}}
}

\maketitle

\begin{abstract}
We study the interaction between a network designer and an adversary over a dynamical network. The network consists of nodes performing continuous-time distributed averaging. The goal of the network designer is to assist the nodes reach consensus by changing the weights of a limited number of links in the network. Meanwhile, an adversary strategically disconnects a set of links to prevent the nodes from converging. We formulate two problems to describe this competition where the order in which the players act is reversed in the two problems. We utilize Pontryagin's Maximum Principle (MP) to tackle both problems and derive the optimal strategies. Although the canonical equations provided by the MP are intractable, we provide an alternative characterization for the optimal strategies that highlights a connection with potential theory. Finally, we provide a sufficient condition for the existence of a saddle-point equilibrium (SPE) for this zero-sum game.
\end{abstract}


\section{Introduction}
Various physical and biological phenomena where global patterns of behaviour stem from local interactions have been modelled using linear averaging dynamics. In such dynamics an agent updates its value as a linear combination of the values of its neighbors. Averaging dynamics is the basic building block in many multi-agent systems, and it is widely used whenever an application requires multiple agents, who are graphically constrained, to synchronize their measurements. Examples include formation control, coverage, distributed estimation and optimization, and flocking \cite{SaberMurray,SaberFlocking,NedicOzdaglarParrilo}. Besides engineering, linear averaging finds applications in other fields. For instance, social scientists use averaging to describe the evolution of opinions in networks \cite{JacksonGolub}.

In practice, communication among agents is prone to different non-idealities which can affect the convergence properties of distributed algorithms. Transmission delays \cite{NedicOzdaglar}, noisy links \cite{XiaoBoyd,TouriNedic}, and quantization \cite{KashyapBasarSrikant} are some examples of non-idealities that are due to the physical nature of the application. In addition to physical restrictions, researchers have also studied averaging dynamics in the presence of malicious nodes in the network \cite{Sundaram,SandbergJohansson}. Various algorithms that guarantee resilience against node failures have been proposed in the literature \cite{pease1980reaching}. 

In \cite{KhanaferTouriBasarNecSys12}, we explored the effect of an external adversary who attempts to prevent the nodes from reaching consensus by disconnecting certain links in the network. We derived his optimal strategy and demonstrated that it admits a potential-theoretic analogy. Here, we introduce a network designer that attempts to counter the effect of the adversary and help the nodes reach consensus by changing the weights of certain links. Both the adversary and the designer are constrained by their physical properties, e.g., battery life and communication range. To capture such constraints, we allow the adversary and the designer to affect  only a fixed number of links. 

Such an interaction between a network designer and an adversary can occur in various practical applications. For example, in a wireless network, the adversary can be a jammer who is capable of breaking links by injecting high noise signals that disrupt the communication among nodes. The link weights in such a network represent the capacities of the corresponding links. The designer can modify the capacity of a certain link using various communication techniques such as introducing parallel channels between two nodes as in orthogonal frequency division multiple access (OFDMA) networks \cite{tse2005fundamentals}.

Our model is different from the current literature in two ways: (i) the adversary and the designer compete over a dynamical network. This is different from the problems studied in the computer science and economics communities where the network is usually static \cite{Goyal2010Robust}; (ii) the players in our model are constrained and cannot operate with an infinite budget. This enables us to model practical scenarios more closely rather than allowing the malicious behaviour to be unrestricted as in \cite{LeBlanc,BicchiBullo,Sundaram}.

The main goal of this work is to derive optimal strategies for the designer and the adversary who have conflicting objectives. Because the order in which the players act affects the resulting utilities, we formulate two problems based on the order of play, allowing each player to have the \emph{first-move-advantage} in one of the problems. When the adversary is allowed to play first, he is capable of restricting the available actions of the designer since some links will disappear from the network. Hence, if we were to cast a zero-sum game between the players, one should not expect the existence of a saddle-point equilibrium (SPE) in pure strategies. The question we would like to answer is then: \emph{are there scenarios where the order of play does not affect the utilities of the players, which leads to the existence of an SPE?} 

Accordingly, the contributions of this paper are as follows.
\begin{itemize}
\item We capture the interaction between the designer and the adversary by formulating two problems. In the min--max problem, the designer declares a strategy first to which the adversary reacts by its optimal response. The second problem is a max--min one where the order of play is reversed. We derive the optimal strategies in both problems by employing the Maximum Principle (MP).
\item We provide a method to compute the optimal strategies without requiring the adjoint equations to be solved. This method provides a new characterization for the optimal strategies in terms of potential-theoretic quantities.
\item We derive a sufficient condition guaranteeing the existence of an SPE.
\end{itemize}

The rest of the paper is organized as follows. In Section \ref{ProbDesc}, we formulate and provide the preliminaries of the min--max and max--min problems. We analyze both problems in Section \ref{Problems} and derive their solutions. An implementation method is provided in Section \ref{implement}. We provide a sufficient condition for the existence of an SPE in Section \ref{suffCond}, and we conclude the paper in Section \ref{Conclusion}. An Appendix contains two lemmas used in the derivation of the sufficient condition.

\emph{Notation}: We denote the set of edges in a graph $\mG$ by $\mE(\mG)$. When clear from the context, we will drop the argument of any set defined on a graph. To emphasize the effect of link removal by the adversary, we will sometimes write $\mG(u(t))$ to denote the graph resulting after the adversary acts at time $t$. We will use $\sum_{j > i} (.)$ to mean $\sum_{j=2}^{n}\sum_{i=1}^{j-1}(.) $. 

\section{Problem Description} \label{ProbDesc}
Consider a connected network of $n$ nodes described by an undirected graph  $\mathcal{G}$ with $n$ vertices and $m$ edges. We will denote an edge in $\mathcal{E}(\mG)$ between nodes $i$ and $j$ by $\{i,j\}$. The value, or state, of the nodes at time instant $t$ is given by $x(t) = [x_1(t),...,x_n(t)]^T$. The nodes start with an initial value $x(0)=x_0$, and they are interested in computing the average of their initial measurements $x_{avg} = \frac{1}{n}\sum_{i=1}^n x_i(0)$ via local averaging. We consider the continuous-time averaging dynamics given by
\begin{equation} \label{systemEqn}
\dot{x}(t) = A(t)x(t), \quad x(0) = x_0,
\end{equation}
where the matrix $A(t)$ has the following properties
\[
A(t) = A(t)^T, \quad A_{ij}(t) \geq 0\quad  i \neq j, \quad A(t)\mathbf{1} = 0.
\]
Define $\bar{x} = \mathbf{1}x_{avg}$ and let $M = \frac{\mathbf{1}\mathbf{1}^T}{n}$. A well-known result states that, given the above assumptions, the nodes will reach consensus as $t\to \infty$, i.e.,  $\lim_{t\to \infty} x(t) =\bar{x}$ \cite{SaberMurray}. In the absence of the adversary and the network designer, $A(t)$ is time-invariant with $A_{ij}(t)=a_{ij}$. To achieve their respective objectives, the designer and the adversary control the elements of $A(t)$ as we explain next. 

The adversary attempts to slow down convergence by breaking at most $\ell \leq m$ links at each time instant. Let $u_{ij}(t) \in \{0,1\}$ be the weight the adversary assigns to link $\{i,j\}$. He breaks link $\{i,j\}$ at time $t$ when $u_{ij}(t) = 1$. His control is given by $u(t)=[u_{12}(t),u_{13}(t),...,u_{1n}(t),u_{23}(t),...,u_{(n-1)n}(t)]^T$. We will denote the number of links the adversary breaks at time $t$ by $N_u(t)$.  In accordance with the above, the strategy space of the adversary is
\[
\mathcal{U} = \left\{u \in \{0,1\}^{n\choose2}: N_u \leq \ell \right\}.
\]
Meanwhile, the network designer attempts to accelerate convergence by controlling the weights of the edges. Let $v(t)=[v_{12}(t),v_{13}(t),...,v_{1n}(t),v_{23}(t),...,v_{(n-1)n}(t)]^T$ be the control of the designer. The designer can change the weight of a given link by adding $v_{ij}(t)$ to its weight $a_{ij}$. We assume that $v_{ij}(t) \in [0,b]$ and that the number of links the designer modifies $N_v(t)$ is at most $\ell$. The strategy set of the designer is therefore
\[
\mathcal{V} = \left\{v\in [0,b]^{n\choose2}: N_v \leq \ell \right\}.
\] 
If $\{i,j\} \notin \mathcal{E}$, then $u_{ij}(t) = v_{ij}(t) = 0$ for all $t$. Given the above definitions, we can write the $\{i,j\}$-th element of the matrix $A(u(t),v(t))$ as 
\[
A_{ij}(u(t),v(t)) = (a_{ij} + v_{ij}(t)) \cdot (1-u_{ij}(t)).
\]
Implicit in this definition is the assumption that the effect of the adversary and the designer on the links lasts only instantly, and the weights of the links are reset to $a_{ij}$ at the next time instant. This is motivated by applications where the state of a channel deteriorates only when the adversary is acting, e.g., wireless communications where the effect of the jammer is not permanent.

Define the following functional:
\begin{equation*}
J(u,v) = \int_0^T k(t)\left| x(t)-\bar{x}\right|^2 dt,
\end{equation*}
where the kernel $k(t)$ is positive and integrable over $[0,T]$. The utility of the adversary is then given by $J(u,v)$, and that of the designer is $-J(u,v)$. We will study the following two problems. In the first problem, the adversary acts first by selecting the links he is interested in breaking. Then, the network designer optimizes his choices over the resulting graph $\mathcal{G}(u(t))$. More formally, we have
\begin{eqnarray*}
&\max\limits_{u(t) \in \mathcal{U}} \min\limits_{v(t) \in \mathcal{V}} & J(u,v) \\
& \text{s.t.} & \dot{x}(t)=A(u(t),v(t))x(t), \quad x(0) = x_0.
\end{eqnarray*}

In the second problem, the order is reversed, and the designer acts first:
\begin{eqnarray*}
&\min\limits_{v(t) \in \mathcal{V}} \max\limits_{u(t) \in \mathcal{U}}  & J(u,v) \\
& \text{s.t.} & \dot{x}(t)=A(u(t),v(t))x(t), \quad x(0) = x_0.
\end{eqnarray*}

In a computer network, where consensus could model the spread of a virus, the max--min problem allows the network architect (who is the maximizer here) to implement networks that are robust against a strategic virus diffusion. The min--max problem finds applications in army combat situations where the designer (the minimizer) attempts to counter the attacks of the enemy to disrupt the communications among agents. In both problems, we make the following assumption:
\begin{assumption}
The initial matrix $A(0,0)$, the state $x(t)$, and the value $b$ are common information to both players. 
\end{assumption}
The following two remarks are in order.
\begin{remark} \label{rem::complex}
\emph{(Problem Complexity)} Let us consider the problem of the adversary for a given strategy of the designer. Assume that the adversary can act in $K$ time instances over the interval $[0,T]$. Then, for small $\ell$, the total number of links that need to be tested in a brute-force approach is
\begin{equation}\label{bruteForce}
{m\choose{\ell}}^K \approx m^{\ell K}.
\end{equation}
Clearly, the brute-force approach incurs exponential-time complexity, especially if any of $m,\ell,k$ becomes large. The same argument applies to the network designer.
\end{remark}
\begin{remark}
\emph{(Non-Rectangular Strategies)} When the strategy sets are rectangular, i.e., the strategy of one player does not restrict the strategy space of the other, the following holds
\begin{equation} \label{eqn::upperlowerV}
\underline{V} = \max\limits_{u(t) \in \mathcal{U}} \min\limits_{v(t) \in \mathcal{V}} J(u,v) \leq \min\limits_{v(t) \in \mathcal{V}}  \max\limits_{u(t) \in \mathcal{U}} J(u,v) = \overline{V},
\end{equation}
where $\underline{V}, \overline{V}$ are called the lower and upper values of the game, respectively. An SPE in pure strategies exists when $\underline{V} = \overline{V}$. When the strategy spaces are non-rectangular, one should \emph{not} expect the existence of an SPE in pure-strategies nor the order in (\ref{eqn::upperlowerV}). In the max--min problem, the adversary restricts the strategy space of the designer since $u(t)$ removes links from $\mG$. Hence, the max--min and min--max problems should be studied separately.
\end{remark}
\section{Adversary vs. Network Designer: Two Problems} \label{Problems}
We will now present the solutions for the two problems presented above. To arrive at the optimal strategies of the players, we employ the MP. In Section \ref{implement}, we prove that $N_u^\star = N_v^\star = \ell$; we will use this fact in the proofs of this section. In what follows, we will often drop the time index and other arguments for notational simplicity. 
\subsection{The Min--Max Problem}
The Hamiltonian associated with the max--min problem is:
\[
H(x,p,u,v) = k(t)\left|x(t)-\bar{x}\right|^2 + p^T(t)A(t)x(t).
\]
The first-order necessary conditions for optimality are:
\begin{eqnarray}
\dot{p} & = & -\frac{\partial }{\partial x}H \nonumber \\
& = & -2k(x-\bar{x}) - Ap, \quad p(T) = 0 \label{eqn::ODEp1} \\
\dot{x} & = & Ax, \quad x(0) = x_0 \label{eqn::ODEx1} \\
(u^\star,v^\star) & = & \arg \max \arg \min \left\{H : u \in \mathcal{U}, v \in \mathcal{V}\right\}, \nonumber
\end{eqnarray}
where $x,p \in C^1[0,T]$, the space of continuously differentiable functions over $[0,T]$.

To find the optimal strategies, let us first write
\begin{eqnarray*}
p^TAx & = & \sum_{i=1}^n p_i \left(\sum_{j=1}^n A_{ij}x_j \right) \\
& = &  \sum_{i=1}^n p_i \left(- \sum_{j=1,j\neq i}^n A_{ij}x_i + \sum_{j=1, j\neq i}^n A_{ij}x_j \right) \\
& = & \sum_{j>i} (a_{ij}+v_{ij})(1-u_{ij})(p_j-p_i)(x_i - x_j).
\end{eqnarray*}
Define the function
\[
f_{ij} = (p_j-p_i)(x_i - x_j),
\]
and let us write
\begin{eqnarray}
 \max_{u \in \mathcal{U}} \min_{v \in \mathcal{V}} H & = & \max_{u \in \mathcal{U}} \min_{v \in \mathcal{V}} k\left|x-\bar{x}\right|^2 + p^TAx \nonumber\\
& = & k\left|x-\bar{x}\right|^2 + \max_{u \in \mathcal{U}} \min_{v \in \mathcal{V}} \sum_{j>i}  A_{ij}f_{ij}. \label{eqn::maxminH}
\end{eqnarray}
Note that we cannot decouple the maximization (or minimization) into $n\choose{2}$ maximization problems, each corresponding to a link or a single term inside the double summation. This is due to the constraint on the number of links that can be targeted by the players. Before finding the optimal strategies, let us first define the following sets. Let $\mathcal{F}(\mG') = \{f_{ij}<0: \{i,j\} \in \mathcal{E}(\mG')\}$ for some graph $\mG'$. It should be noted that although $\{i,j\}$ and $\{j,i\}$ belong to $\mathcal{E}(\mG')$, we only include $f_{ij}$ once in $\mathcal{F}(\mG')$, i.e., we treat $\{i,j\}$ and $\{j,i\}$ as one link. This applies to all the definitions to follow. Let $\mathcal{F}_{\pi}(\mG')$ be a ranking of the set $\mathcal{F}(\mG')$ such that $\mathcal{F}_{\pi,1}(\mG') \leq \hdots \leq \mathcal{F}_{\pi,|\mF(\mG')|}(\mG')$. Further, let $\mathcal{F}_{\pi}^\ell(\mG')$ be the set containing the smallest $\ell$ values in $\mathcal{F}_{\pi}(\mG')$. It is understood that if $\mF_\pi(\mG')$ contains fewer than $\ell$ negative elements, $\mF_\pi^\ell$ will contain all those elements. Define the set operator $\Phi_i:S(\mG') \to \mE(\mG')$ that returns the links in $\mE(\mG')$ that correspond to the elements of $\mS(\mG')$. Also, define $\Phi_i: S(\mG') \to \mE(\mG')$ that returns the links in $\mE(\mG')$ corresponding to the \emph{smallest} $i$ elements of the set $S(\mG')$. When $|S(\mG')| < i$, $\Phi_i(S(\mG')) = \Phi(S(\mG'))$. We also adopt the convention $\Phi_0(.) = \{ \emptyset \}$. We can then write 
$$
\mathcal{F}_{\pi}^\ell(\mG') = \Phi_{\ell}\left(\left\{f_{ij}\leq \mathcal{F}_{\pi,\ell+1}(\mG') : \{i,j\}\in \mE(\mG')  \right\}\right).
$$
Define the sets $\mathcal{D}_1(\mG') = \{a_{ij}f_{ij} < 0: \{i,j\}\in \mathcal{E}(\mG') \}$, $\mathcal{D}_2(\mG') = \{(a_{ij}+b)f_{ij} < 0: \{i,j\}\in \mathcal{E}(\mG') \}$, and $\mathcal{D}(\mG') = \mathcal{D}_1(\mG') \cup \mathcal{D}_2(\mG')$. Also, let $\mathcal{D}_{\pi}(\mG')$ be a ranking of $\mathcal{D}(\mG')$ such that $\mathcal{D}_{\pi,1}(\mG') \leq \hdots \leq \mD_{\pi,|\mD(\mG')|}(\mG')$. Finally, define 
\begin{eqnarray*}
\mathcal{D}_{\pi}^\ell(\mG') & = & \Phi_\ell(\left \{a_{ij}f_{ij}\leq \mathcal{D}_{\pi,\ell+1}(\mG') \text{ or } \right. \\ && \left. (a_{ij}+b)f_{ij}\leq \mathcal{D}_{\pi,\ell+1}(\mG') : \{i,j\}\in \mE(\mG') \right\}).
\end{eqnarray*}

The following theorem specifies the optimal strategies of the adversary and the designer.

\begin{theorem} \label{thm::maxmin}
The optimal strategies for the network designer and the adversary in the max--min problem are\begin{eqnarray*}
v^\star_{ij}(t)  & = &   \left\{
  \begin{array}{l l}
    b, &  \text{$\{i,j\} \in \mathcal{F}^{\ell}_\pi(\mG^\star)$}  \\
    0, &  \text{$\{i,j\} \notin \mathcal{F}^{\ell}_\pi(\mG^\star)$ or $f_{ij} > 0$}  \\
    \text{$[0,b]$}, &  \text{otherwise}
  \end{array} \right. \\
u^\star_{ij}(t) & = & \left\{
  \begin{array}{l l}
    1, &  \text{$\{i,j\} \in \mD_{\pi}^{\ell}(\mG)$}  \\
    0, &  \text{$\{i,j\} \notin  \mD_{\pi}^{\ell}(\mG)$ or $f_{ij} > 0$} \\
    \{0,1\}, & \text{otherwise}
  \end{array} \right.  
\end{eqnarray*}
where $\mG^\star = \mG(u^\star)$ is the graph resulting after the adversary applies $u^\star$.
\end{theorem}
\begin{proof}
For both strategies, when $f_{ij}=0$, the MP does not tell us what the optimal action is. Hence, both players can select any action from their strategy spaces in this case, provided they do not exceed their budget constraints. Let us start by deriving $v^\star$. By (\ref{eqn::maxminH}), the terms $a_{ij}f_{ij}$ do not affect the minimization problem. Also, note that $\mG^\star$ contains the links for which $u_{ij}^* = 0$ only. Then, we can write
\begin{equation} \label{indepOfAij}
\underset{{v\in \mV}}{\operatorname{argmin}} \sum_{j>i}  A_{ij}f_{ij} = \underset{{v\in \mV}}{\operatorname{argmin}}\sum_{\substack{j>i \\ \{i,j\}\in \mE(\mG^\star)}}  v_{ij}f_{ij}.
\end{equation}
Now consider the sum on the right hand side. Being a minimizer, the designer should not amplify the positive $f_{ij}$'s; hence, he must set $v^\star_{ij}=0$ when $f_{ij}>0$. Then, by recalling the definition of $\mF_{\pi}^\ell(\mG^\star)$, we obtain
\begin{eqnarray*}
\sum_{\substack{j>i \\ \{i,j\}\in \mE(\mG^\star)}}  v_{ij}f_{ij} & = & \sum_{\substack{j>i\\ \{i,j\}\in \mF_{\pi}^\ell(\mG^\star)}}  v_{ij}f_{ij} + \sum_{\substack{j>i\\ \{i,j\}\notin \mF_\pi^\ell(\mG^\star)}}  v_{ij}f_{ij} \\
&\geq& \sum_{\substack{j>i\\ \{i,j\}\in \mF_{\pi}^\ell(\mG^\star)}}  bf_{ij}.
\end{eqnarray*}
Note that the last inequality can indeed be achieved because $|\mF_\pi^\ell(\mG^\star)| \leq \ell$. Hence, the optimal $v^\star$ is as claimed.

In this problem, the adversary has the first-mover-advantage and needs to dispose of the links that can reduce his utility. Because the adversary attempts to maximize the minimum of $p^TAx$, it follows that $u^\star_{ij} = 0$ when $f_{ij} > 0$. Let us now consider the links with $f_{ij} < 0$. The adversary knows that the designer attempts to make the negative $f_{ij}$'s smaller by adding $b$ to the corresponding edge weights. However, we cannot rule out the possibility that $(a_{lk}+b)f_{lk} > a_{ij}f_{ij}$. Hence, the adversary is not only interested in finding the smallest negative $(a_{ij}+b)f_{ij}$'s, but also needs to consider the $a_{ij}f_{ij}$'s themselves. It follows that the adversary needs to find the terms that can become very small (negative) and set $u_{ij}=1$ to the corresponding links. But those links are exactly the ones included in $\mD_\pi^\ell(\mG)$. Formally, and using the notation $A_{ij}^* = (a_{ij}+v_{ij}^\star)(1-u_{ij})$, we have
\begin{eqnarray*}
\sum_{j>i}  A_{ij}^\star f_{ij}  & = & \sum_{\substack{j>i\\ \{i,j\}\in \mD_{\pi}^\ell(\mG)}}  A^\star_{ij}f_{ij} + \sum_{\substack{j>i\\ \{i,j\}\notin \mD_{\pi}^\ell(\mG)}}  A^\star_{ij}f_{ij}\\
& \leq & \sum_{\substack{j>i\\ \{i,j\}\notin \mD_{\pi}^\ell(\mG)}}  (a_{ij}+v_{ij}^\star)f_{ij},
\end{eqnarray*}
where the last inequality follows from the fact that $f_{ij} < 0$ for all $\{i,j\}  \in  \mD_{\pi}^\ell(\mG)$. This confirms that $u^\star$ is as claimed.
\end{proof}
\subsection{The Min--Max Problem}
We now study the min--max problem. Similar to the above, we can write
\begin{eqnarray*} 
 \min_{v \in \mathcal{V}} \max_{u \in \mathcal{U}}H = k\left|x-\bar{x}\right|^2 +\min_{v \in \mathcal{V}} \max_{u \in \mathcal{U}}  \sum_{j>i}  A_{ij}f_{ij}. 
\end{eqnarray*}
We again note that one needs to be careful to observe that the minimization (or maximization) problem does not decouple due to the constraint on the controls of the two players. Although the adversary's strategy can be derived using an approach similar to the previous problem, the designer's strategy requires more care. Let $\mL(v) = \{(a_{ij}+v_{ij})f_{ij}<0: \{i,j\}\in \mE(\mG) \}$. Define $\mL_{\pi}(v)$ to be a ranking of the set $\mL(v)$ such that $\mL_{\pi,1}(v)\leq \hdots \leq \mL_{\pi,|\mL| }(v)$, and the set
\[
\mL_{\pi}^\ell(v) = \Phi_{\ell}\left( \left\{(a_{ij}+v_{ij})f_{ij}\leq \mL_{\pi,\ell+1}(\mG): \{i,j\}\in \mE(\mG) \right\} \right).
\]
As an abuse of notation, we let $\mL_{\pi,k}^\ell(v)$ correspond to the \emph{value} $(a_{ij}+v_{ij})f_{ij}$ associated with the link $\{i,j\}\in \mL_\pi^\ell(v)$. We assume that $\mL_{\pi,1}^\ell(v) \geq \hdots \geq \mL_{\pi,\ell}^\ell(v)$. Further, define the sets $\mP(v) = \{a_{ij}f_{ij}<0: \{i,j\} \notin \mL_{\pi}^\ell(v)\}$ such that $\mP_1(v)\leq  \hdots \leq \mP_{|\mP|}(v)$ and $\overline{\mP}(v) = \{f_{ij}<0: \{i,j\} \notin \mL_{\pi}^\ell(v)\}$ such that $\overline{\mP}_1(v)\leq  \hdots \leq \overline{\mP}_{|\overline{\mP}|}(v)$. We also define
\begin{eqnarray*}
[v_{\mS}(b)]_{ij} & = & \left\{
  \begin{array}{l l}
    b, &  \text{$\{i,j\} \in \mS$}  \\
    0, &  \text{$\{i,j\} \notin  \mS$}
  \end{array} \right.  
\end{eqnarray*}
The following theorem specifies the optimal strategies of both players.
\begin{theorem}
The optimal strategy for the adversary in the min--max problem is
\begin{eqnarray*}
u^\star_{ij}(t) & = & \left\{
  \begin{array}{l l}
    1, &  \text{$\{i,j\} \in \mL_{\pi}^{\ell}(v^\star)$}  \\
    0, &  \text{$\{i,j\} \notin  \mL_{\pi}^{\ell}(v^\star)$ or $f_{ij} > 0$} \\
    \{0,1\}, & \text{otherwise}
  \end{array} \right.
\end{eqnarray*}  
For the network designer, the optimal strategy is to run Algorithm \ref{algorithm} and set $v_{ij}^\star =0$ when $f_{ij}>0$ and $v_{ij}^\star \in [0,b]$ if $f_{ij} = 0$. 
\end{theorem}
\begin{table}[!t]
 \centering
 \caption{
   Algorithm I: Computing the optimal strategy for the minimizer in the min--max problem.}\vspace*{-1em}
   \begin{tabular}{p{7cm}}
     \hline \\
0:\hspace*{1em} \textbf{input:} \text{a strategy $v$ with $N_v = 0$}\vspace*{.2em}\\
1:\hspace*{1em} \textbf{for $i = \ell \downarrow  1$}\vspace*{.2em}\\
2: \hspace*{2em} \textbf{if} $\exists \mS \subset \mP(0), |\mS| = i$, $\mL_{\pi,i}^\ell(0) \notin \mL_{\pi}^{\ell}(v_\mS(b))$\vspace*{.1em}\\
3:\hspace*{3em} Set $v_{ij} = b$, $\forall \{i,j\} \in \Phi(\mS) \cup \Phi_{\ell-i}\left(\overline{\mP}(v_\mS(b))\right)$.\vspace*{.1em}\\
4:\hspace*{3em} \textbf{Exist} for loop.\vspace*{.1em}\\
5:\hspace*{2em}\textbf{end}\vspace*{.1em}\\
6:\hspace*{1em}\textbf{end}\vspace*{.1em}\\ 
7:\hspace*{1em}\textbf{if} $N_v =0$ \vspace*{.1em}\\
8:\hspace*{2em} Set $v_{ij} = b$ for all $\{i,j\} \in \Phi_{\ell}\left(\overline{\mP}(0)\right)$.\vspace*{.1em}\\
9:\hspace*{1em}\textbf{end}\vspace*{.1em}\\ \\
  \hline
   \end{tabular}\label{tab:algo}\vspace{-0.5cm}
   \label{algorithm}
\end{table}

\begin{proof}
The same reasoning as in Thm \ref{thm::maxmin} applies to the case when $f_{ij}\geq0$. We start by deriving $u^\star$. By recalling the definition of $\mL_{\pi}^\ell(v^\star)$, we can write
\begin{eqnarray*}
\sum_{j>i}  A_{ij}^\star f_{ij} & = & \sum_{\substack{j>i\\ \{i,j\}\in \mL_{\pi}^\ell(v^\star)}}  A^\star_{ij}f_{ij} + \sum_{\substack{j>i\\ \{i,j\}\notin \mL_{\pi}^\ell(v^\star)}}  A^\star_{ij}f_{ij}\\
& \leq & \sum_{\substack{j>i\\ \{i,j\}\notin \mL_{\pi}^\ell(v^\star)}}  (a_{ij}+v_{ij}^\star)f_{ij},
\end{eqnarray*}
where the last inequality follows from the fact that $f_{ij} < 0$ for all $\{i,j\}  \in  \mL_{\pi}^\ell(v^\star)$. This confirms that $u^\star$ is as claimed. Also, note that the last inequality can be achieved because $|\mL_\pi^\ell(v^\star)| \leq \ell$ by construction. 

As for the designer, he can now exploit the first-mover-advantage since he is acting first. The main goal of the designer in this case is to try to protect the links in $\mL_{\pi}^\ell(0)$. These links will be disconnected by the adversary unless the designer is able to modify the ranking of the links such that the links (or a subset of them) in $\mL_\pi^\ell(0)$ are not in $\mL_\pi^\ell(v^\star)$. In essence, this is what Algorithm \ref{algorithm} attempts to achieve. This is demonstrated in Fig. \ref{fig::deception}.
\begin{figure}[!t]
\centering
\includegraphics[width=8cm]{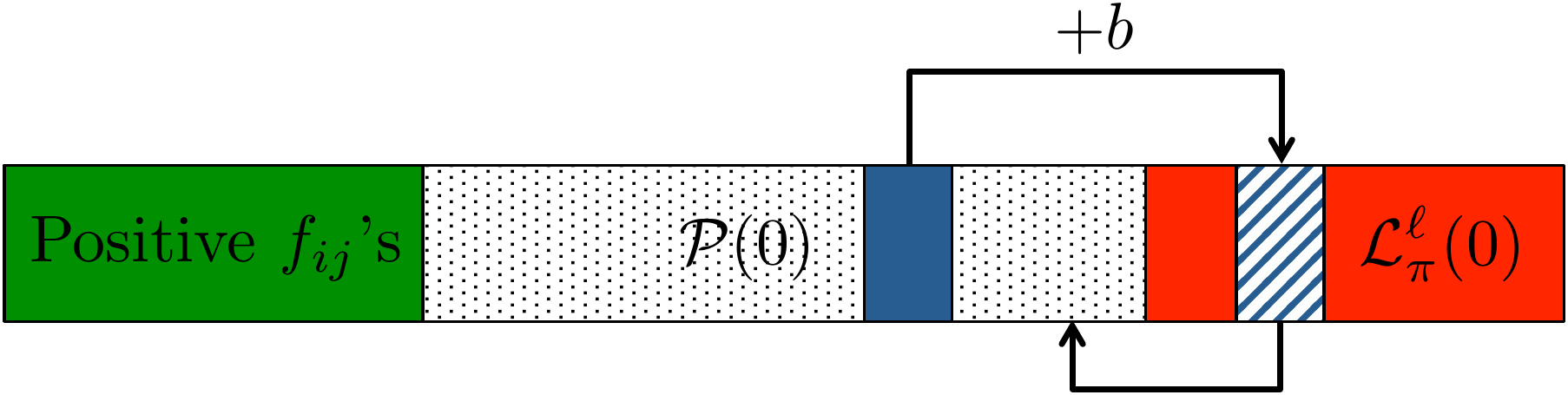}
\caption{A demonstration of how the designer can alter the ranking of the $f_{ij}$'s to improve his utility.}
\label{fig::deception}
\end{figure}
Being the lowest negative value, and hence the link both the adversary and the designer are interested in, let us explore how the designer can push $\mL_{\pi,\ell}^\ell(0)$ higher in the ranking of the link values. The designer can achieve this if under some strategy $v$, the value $\mL_{\pi,\ell}^\ell(0)$ is no longer among the lowest $\ell$ negative values; in other words, the designer can alter the ranking if there is a set of $\ell$ links $\mS \subset \mP(0)$ such that when he sets $v_{ij} = b$ for all links in $\mS$, there will be $\ell$ values that are smaller than $\mL_{\pi,\ell}^\ell(0)$ (steps $2$ and $3$ in Algorithm \ref{algorithm}). The adversary will then break the links in $\mS$ and will spare the link corresponding to $\mL_{\pi,\ell}^\ell(0)$ as required. To see why this is optimal, consider the type of links that can be in $\mS$.
\begin{itemize}
\item If a link in $\mS$ is also in $\mL_\pi^\ell(0)$, then this is optimal due to the fact that the adversary will disconnect that link since it is in $\mL_\pi^\ell(0)$. Hence, if the designer can utilize this link to modify the ranking and protect a link that is more negative ($\mL_{\pi,\ell}^\ell(0)$ in this case), then this can only improve his utility. The same reasoning applies if more than one of the links in $\mS$ are also in $\mL_\pi^\ell(0)$.
\item If none of the links in $\mS$ is in $\mL_\pi^\ell(0)$, then necessarily some of the links in $\mL_\pi^\ell(0)$ will also be protected along with the link corresponding to $\mL_{\pi,\ell}^\ell(0)$. This is because $|\mS| = \ell$, and the adversary can only break $\ell$ links. Hence, this scenario is more favorable to the designer than the previous one and can therefore only improve his utility.
\end{itemize}
If such an $\mS$ exists, then the designer has exhausted all possible moves, since $|\mS| = \ell$, and the algorithm terminates (step $4$ of the algorithm). Otherwise, If no such set exists in $\mP(0)$, then the designer should try to protect the next most negative link whose value is precisely $\mL_{\pi,\ell-1}^\ell(0)$ by finding a set $\mS$ of of size $\ell-1$. Since $\mL_{\pi,\ell-1}^\ell(0) \geq \mL_{\pi,\ell}^\ell(0)$, the link corresponding to $\mL_{\pi,\ell}^\ell(0)$ along with $\mS$ will constitute the set of $\ell$ links that the adversary will break. Then, the designer should set $v_{ij}=b$ for all the links in $\mS$, and for the remaining action the designer should select the link with the most negative $f_{ij}$ that is \emph{not} in $\mL_\pi^\ell(v_\mS(b))$; this is precisely the set $\Phi_{1}\left(\overline{\mP}(v_\mS(b))\right)$ (step $3$ of the algorithm). The reason behind searching in $\overline{\mP}(v_\mS(b))$ and not in $\mP(v_\mS(b))$ after finding $\mS$ is that the $a_{ij}$'s only affect the utility of the designer when he attempts to alter the ranking. Otherwise, the edge weights do not affect the utility of the minimizer as per (\ref{indepOfAij}). 

This procedure then repeats until the designer has tried to protect all the links in $\mL_\pi^\ell(0)$. If the designer fails in protecting \emph{all} the links in $\mL_\pi^\ell(0)$, then we must have $N_v = 0$, i.e., the input strategy was not altered. Then, the optimal strategy is to set $v_{ij} =b$ for the links with most negative $f_{ij}$'s in $\overline{\mP}(0)$ (steps $7$ and $8$ in Algorithm \ref{algorithm}).
\end{proof}

\subsection{Complexity of the Optimal Strategies}
Let us now study the complexity of the optimal strategies. We first start by the max--min problem. Assuming as in Remark \ref{rem::complex} that the players switch their strategies a total of $K$ times over $[0,T]$, we conclude that the worst-case complexity of the strategy of both players is $\mathcal{O}(K\cdot m\log m)$ as their strategies involve merely the ranking of sets of size at most $2m$. As for the min--max problem, the complexity of the adversary's strategy is $\mathcal{O}(K\cdot m\log m)$. The main bottleneck in the strategy of the designer is step $2$ in Algorithm \ref{algorithm}. The size of the set $\mP(0)$ is at most $m-\ell$; thus, the worst-case complexity for the designer is $K \cdot \sum_{i=1}^{m-\ell}{m-\ell \choose i} \approx K\cdot \sum_{i=1}^\ell(m-\ell)^i$. By comparison with (\ref{bruteForce}), we conclude that the derived optimal strategies achieve vast complexity reductions.

\section{Implementing the Optimal Strategies} \label{implement}
The optimal strategies we derived in Section \ref{Problems} are defined in terms $f_{ij}$'s which depend on the state $x(t)$ and the costate $p(t)$. However, we have not derived the optimal trajectories that satisfy the canonical equations given by the MP in (\ref{eqn::ODEp1}) and (\ref{eqn::ODEx1}). Since the system is linear time-varying, the solutions will be given in terms of a state transition matrix; this makes working with $f_{ij}$ intractable. The following theorem provides a procedure to arrive at the optimal solutions without the need to compute $p(t)$. Parts of the proof of this theorem appeared in \cite{KhanaferTouriBasarNecSys12}. Define $\nu_{ij} = -(x_i-x_j)^2$. 
\begin{theorem} \label{thm::implement}
The rankings performed as part of the optimal strategies of the max--min and min--max problems can be carried out by replacing $f_{ij}(t)$ by $\nu_{ij}(t)$. Further, it is optimal for the players to modify a total of $\ell$ links.
\end{theorem}
\begin{proof}
We will provide the main ideas behind proving this claim for the max--min problem -- the same arguments will apply to the min--max problem. For a fixed strategy of the adversary $u^\star$, we will show that it is optimal for the minimizer to rank the links based on their $\nu_{ij}$ values instead of the $f_{ij}$'s. Then, by applying the same reasoning as in Thm \ref{thm::maxmin}, we arrive at $u^\star$ with the ranking performed using $\nu_{ij}$. 

The main complication in solving the adjoint equations is that the system is time-varying. However, since $x(t),p(t) \in C^1[0,T]$, the value of $f_{ij}$ cannot change abruptly in a finite interval. As a result, the control obtained from the MP cannot switch infinitely many times in a finite interval. Hence, we can regard the system as a time-invariant one over a small interval $[t_0,t_0+\delta] \subset [0,T]$, $\delta > 0$. The proof consists of three steps.
\begin{itemize}
\item[1.] Show that it is optimal for the players to change $\ell$ links.
\item[2.] Show that, over a small interval $[t_0,t_0+\delta]$, it is optimal for the designer to switch from a strategy $v^A$ to a strategy $v^B$ that entails ranking the links based on their $\nu_{ij}$ values.
\item[3.] Show that mimicking $v^A$ for the remaining time of the problem preserves the gain obtained over $[t_0,t_0+\delta]$.
\end{itemize}
The first step was proven in \cite{KhanaferTouriBasarNecSys12} for the maximization problem. In the min--max (or max--min) case, the proof remains the same. This is because for a fixed $u^\star$, then by Thm 1 in \cite{KhanaferTouriBasarNecSys12}, it follows that it is optimal for the designer to modify $\ell$ links\footnote{Conditions were given in \cite{KhanaferTouriBasarNecSys12} as to when it is also optimal for a player to change less than $\ell$ links.}. By Thm \ref{thm::maxmin}, we deduce that $N_{u^\star} = \ell$.

We want to show that the designer will modify the $\ell$ links with the lowest $\nu_{ij}$ values. Assume that the designer applies a stationary strategy $v^A$ over the interval $[t_0,t_0+\delta]$ with a corresponding system matrix $A$. Let $P(t):=e^{At}$. Due to the structure of $A$, $P(t)$ is a doubly stochastic matrix for $t \geq 0$. Because the control strategies are time-invariant, the state trajectory is given by
\[
x(t) = e^{A(t-t_0)}x(t_0), \quad t \in [t_0,t_0+\delta].
\]
Assume that the links the designer modifies over this interval are not the ones with the lowest $\nu_{ij}$ values. In particular, assume that the designer sets $v^A_{ij}=0$ and $v^B_{kl} =b$ although $\nu_{ij} < \nu_{kl}$. At time $t^\star \in [t_0,t_0+\delta]$, we assume that the designer switches to strategy $v^B$ by setting $v^B_{ij}=b$ and $v^B_{kl} =0$. Let the matrix $B$ be the system matrix corresponding to $v^B$, and define the doubly stochastic matrix $Q(t) := e^{Bt}$, $t\geq 0$. Showing that this switch can improve the utility of the designer is equivalent to proving the following inequality \cite{KhanaferTouriBasarNecSys12}:
\begin{equation}
\int_{t^\star}^{t_0+\delta} k(t)\cdot x(t_0)^T\Lambda(t,t^\star)x(t_0) dt < 0, \label{ineq::toProve}
\end{equation}
where $\Lambda(t,t^\star)= P(t^\star-t_0)Q(2(t-t^\star))P(t^\star-t_0) - P(2(t-t_0))$. A sufficient condition for (\ref{ineq::toProve}) to hold is
\begin{equation*}
h(t,x(t_0)) = x(t_0)^T\Lambda(t,t^\star)x(t_0)<0, \text{ for } t> t^\star.
\end{equation*}
As $t \downarrow 0$, we can write $P(t) = I + tA+\mathcal{O}\left(t^2\right)$, where $\mathcal{O}\left(t^2\right)/t \leq L$ for some finite constant $L$. Following similar steps to those in \cite{KhanaferTouriBasarNecSys12}, we can write
\begin{eqnarray*}
h(t,x(t_0)) & = & 2(t-t^\star) \sum_{r>s}(A_{sr} - B_{sr})\left(x_r(t_0)-x_s(t_0)\right)^2 \nonumber \\
& = & 2b(t-t^*)(\nu_{ij}(t_0) - \nu_{kl}(t_0)) + \mathcal{O}(\delta^2). 
\end{eqnarray*}
For small enough $\delta$, the higher order terms are dominated by the first term. Hence, if there are links $\{i,j\}$ and $\{k,l\}$ such that $\nu_{ij}\leq\nu_{kl} $, there exists $t^\star$ such that $h(t,x(t_0))<0$ for $t \in \left(t^\star,t_0+\delta\right]$. This proves that the optimal strategy for the designer is to break the links with the lowest $\nu_{ij}$ values. By comparison with the statement of Thm \ref{thm::maxmin}, we conclude that we can replace $f_{ij}$ by $\nu_{ij}$ while performing the required rankings.

The final step of the proof is to show that switching to strategy $B$ guarantees an improved utility for the adversary \emph{regardless of how the original trajectory $A$ changes beyond time $t_0+\delta$}. To this end, we will assume that from time $t_0+\delta$ onward, strategy $B$ will mimic strategy $A$. Assume that strategy $A$ switches to another strategy $C$ starting at time $t_0+\delta$. Hence, strategy $B$ will also switch to strategy $C$, but the trajectories will have a different initial conditions. Fig. \ref{fig::proofSketch} illustrates this idea. Consider the behavior of the system over $[t_0+\delta, t_0 + 2\delta]$ over which we can assume that the system is time-invariant. One can show that it suffices to prove that \cite{KhanaferTouriBasarTAC13}
\begin{eqnarray*}
\int_{t_0+\delta}^{t_0 + 2\delta} k(t)\cdot c(t) dt < 0,
\end{eqnarray*}
where $c(t)$ is the difference between the utilities resulting from the trajectory that includes a switch to strategy $B$ at time $t^\star$ and the one that does not switch at time $t^\star$. After performing a first-order Taylor expansion, $c(t)$ can be shown to be equal to \cite{KhanaferTouriBasarTAC13}
\[
c(t)= 2b\left(t_0+\delta -\ts\right)(\nu_{ij}(t_0)-\nu_{kl}(t_0)) + \mathcal{O}(\delta^2),
\]
which is negative for small $\delta$. Hence, the gain that was obtained by switching to strategy $B$ at time $t^\star$ is preserved and the proof is complete.
\begin{figure}[!t]
\centering
\includegraphics[width=8cm]{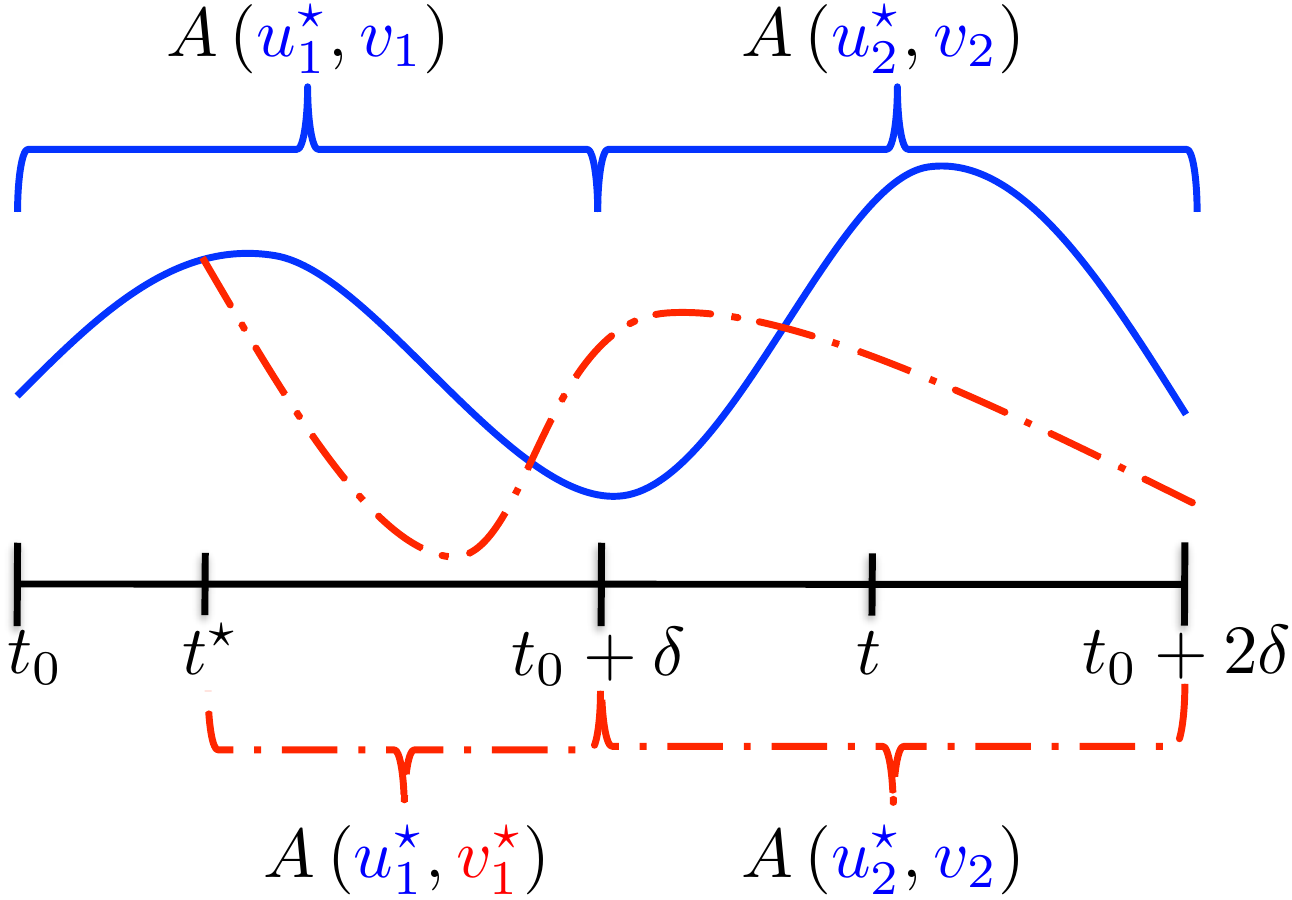}
\caption{A demonstration of the technique used in the third step of the proof.}
\label{fig::proofSketch}
\end{figure}
\end{proof}
\begin{remark}
\emph{(Potential-Theoretic Analogy)} As discussed in \cite{KhanaferTouriBasarNecSys12}, when the graph is viewed as an electrical network, $a_{ij}+v_{ij}$ can be viewed as the conductance of link $\{i,j\}$ and $x_i-x_j$ as the potential difference across the link. Therefore, the optimal strategy of the designer in both problems involves finding the links with the highest potential difference (or the lowest $\nu_{ij}$'s) and increasing the conductance of those links by setting $v_{ij} = b$. This leads to increasing the power dissipation across those links, which translates to increasing the information flow across the network and results in faster convergence. The optimal strategy of the adversary should therefore involve breaking the links with highest power dissipation. But power dissipation is given by $(a_{ij}+v_{ij})(x_i-x_j)^2$, and this is exactly what the adversary targets according to Thm \ref{thm::implement}.
\end{remark}
\section{A Sufficient Condition for the Existence of an SPE} \label{suffCond}
Having solved the min--max and max--min problems, it remains to verify whether the value of these problems can be equal, even though the strategy sets of the players are non-rectangular. The reason that the upper and lower values are different is mainly due to the ability of the minimizer to \emph{deceive} the maximizer by altering the ranking of the most negative values. If we remove this ability from the network designer, we should expect that an SPE would exist. The following theorem makes this argument formal. Define $\gamma :=\frac{M}{\epsilon^2}$, where $M$, $\epsilon$ are defined in Lemmas \ref{lemma::upperBound} and \ref{lemma::lowerBound} in the Appendix. We assume that $\epsilon$ is chosen to guarantee $\gamma > 1$.
\begin{theorem}
A sufficient condition for the upper and lower values $\overline{V},\underline{V}$, defined in (\ref{eqn::upperlowerV}) to be equal, and hence for the existence of an SPE, is to select $b$ such that
\begin{equation} \label{eqn::sufCond}
0 \leq b \leq \min_{\{i,j\}, \{k,l\} \in \mE} \left| \gamma a_{ij} - a_{kl} \right|,
\end{equation}
given that $a_{ij} \neq a_{kl}$ and $a_{ij} > \gamma a_{kl}$ whenever $a_{ij} > a_{kl}$, for all $\{i,j\},\{k,l\}\in \mE$.
\end{theorem}
\begin{proof}
It suffices to show that $\mL_{\pi}^\ell(v^\star) = \mL_{\pi}^\ell(0) = \mD_\pi^\ell(G^\star)$ as this would imply that the adversary would break the same links whether he acts first or second, and as a result the strategy of the minimizer in both problems will be the same. This would occur if the minimizer cannot protect any of the links in $\mL_{\pi}^\ell(0)$. In other words, this will occur if the minimizer cannot satisfy the condition in step $2$ of Algorithm \ref{algorithm} for any $1\leq i \leq \ell$. A sufficient condition for this to happen is to require
\begin{eqnarray*}
\min_{\{i,j\}\in \mP(0)}(a_{ij}+b)f_{ij} & > & \max_{\{i,j\}\in \mL_{\pi}^\ell(0)}a_{ij}f_{ij}.
\end{eqnarray*}
This implies that no matter how the designer changes the weights of the links in $\mP(0)$, it cannot make those links more negative than the links in $\mL_{\pi}^\ell(0)$. To satisfy this inequality, we will establish that whenever $a_{ij}f_{ij} > a_{kl}f_{kl}$, $f_{ij},f_{kl}<0$, we must have $(a_{ij}+b)f_{ij} > a_{kl}f_{kl}$, for all links in $\mE$. We can then re-write the condition on $b$ as
\begin{eqnarray} \label{eqn::condOnb}
b \leq \frac{a_{ij}f_{ij} - a_{kl}f_{kl}}{-f_{ij}} = a_{kl} \frac{|f_{kl}|}{|f_{ij}|} - a_{ij}.
\end{eqnarray}
Consider the following two cases. If $f_{kl}\geq f_{ij}$, then we must have $a_{kl}>a_{ij}$. Then, by assumption we have that $a_{kl} > \gamma a_{ij}$. By Lemmas \ref{lemma::upperBound} and \ref{lemma::lowerBound}, we can write
\begin{equation} \label{condB}
a_{kl} \frac{|f_{kl}|}{|f_{ij}|} - a_{ij}  \geq \frac{1}{\gamma} a_{kl} - a_{ij}>0.
\end{equation}
Next, consider the case when $f_{ij}>f_{kl}$. In this case, $a_{ij}$ can be greater or smaller than $a_{kl}$. However, if $a_{ij}>a_{kl}$, and recalling that $a_{ij}f_{ij} > a_{kl}f_{kl}$, then 
\begin{equation*}
\gamma a_{kl} < a_{ij} < a_{kl}\frac{|f_{kl}|}{|f_{ij}|}\leq \gamma a_{kl},
\end{equation*}
which is a contradiction. The case $a_{kl}=a_{ij}$ is excluded by assumption. Hence, in this case, we must have $a_{ij}<a_{kl}$, and the inequality in (\ref{condB}) applies. Thus, by choosing $b$ as in (\ref{eqn::sufCond}), we obtain the condition we are seeking. Note that we do not need to consider the case when $a_{ij}f_{ij} = a_{kl}f_{kl}$ since the players will be indifferent as to which link to choose when the corresponding values are equal. 
\end{proof}

\begin{remark}
The condition derived in the above theorem requires the network to be \emph{sufficiently diverse} in the sense that the weights of the links have to be not only different from each other, but also a factor $\gamma$ apart. This is due to the fact that we were seeking uniform bounds on the $f_{ij}$'s. If we allow $b$ to vary with time, then one can find less restrictive conditions to ensure the existence of an SPE. However, this would require (\ref{eqn::condOnb}) to be verified at each time instant. Further, the bounds derived in the Appendix are loose. Tighter bounds could be given by studying the dynamics of $|f_{kl}|/|f_{ij}|$. However, studying the time derivative of this ratio is not tractable.
\end{remark}

\section{Conclusion} \label{Conclusion}
In this paper, we have formulated and solved two problems that capture the competition between an adversary and a network designer in a dynamical network. We considered practical models for the players by constraining their actions along the problem horizon. The derived strategies were shown to exhibit a low worst-case complexity. We also proved that the optimal strategies admit a potential-theoretic analogy, and they can be implemented without needing to solve the adjoint equations. Finally, we showed that when the link weights are sufficiently diverse an SPE exists for the zero-sum game between the designer and the adversary.


\bibliographystyle{IEEEtran}
\bibliography{references}

\appendix
\begin{lemma} \label{lemma::upperBound}
The function $f_{ij}(t)$ is uniformly bounded from above by 
\begin{equation*}
M :=4nT\left(\vnorm{x(0)}_\infty+|x_{avg}|\right)\vnorm{x(0)}_{\infty}.
\end{equation*}
\end{lemma}

\begin{proof}
By the structure of the system matrix in (\ref{systemEqn}), we can deduce that $|x_i-x_j|$ cannot increase as $t\to T$. Thus
\begin{eqnarray*}
|x_i(t)-x_j(t)| & \leq & \max_{1\leq i,j \leq n} |x_i(0) -x_j(0)| \\
& \leq & 2\max_{1\leq i \leq n}|x_i(0)| = 2 \vnorm{x(0)}_{\infty}.
\end{eqnarray*}
In order to bound $|p_i(t)-p_j(t)|$, we consider the interval $[T-\delta,T]$ where $\delta > 0$ is small enough such that the system in (\ref{systemEqn}) is time-invariant. Let the system matrix over this interval be $A$ and define $s:=T-\delta$. The solution to the adjoint dynamics in (\ref{eqn::ODEp1}), (\ref{eqn::ODEx1}) over $[s,T]$ is then given by
\begin{eqnarray}
x(t) & = & e^{A(t-s)}x(s), \label{infSolnX}\\
p(t) & = & 2\int_t^Te^{-A(t-\tau)}(x(\tau)-\bar{x})d \tau. \label{infSolnP}
\end{eqnarray}
Define the doubly stochastic matrix $P(t):=e^{At}$. Then
\begin{eqnarray*}
&& |p_i-p_j| = 2\left|\int_t^T (P_i(\tau-t) - P_j(\tau-t))(x(\tau)-\bar{x})d\tau     \right| \\
&& \leq  2\int_t^T \sum_{k=1}^n \left|(P_{ik}(\tau-t) - P_{jk}(\tau-t))(x_k(\tau)-\bar{x}_k)  \right| d\tau \\
&& \leq 2\int_t^T \sum_{k=1}^n \left|x_k(\tau)-\bar{x}_k  \right| d\tau\\
&& \leq 2\int_t^T (\vnorm{x(\tau)}_1+n|x_{avg}|)  d\tau.
\end{eqnarray*}
Note that over $[s,T]$ we have $\vnorm{x(t)}_1 = \vnorm{P(t-s)x(s)}_1 \leq \vnorm{P(t-s)}_1\vnorm{x(s)}_1 = \vnorm{x(s)}_1$, where the last inequality follows from the fact that $P(t)$ is a stochastic matrix. By repeating the above argument over the interval $[0,s)$, after dividing it into smaller intervals and exploiting the fact that the solution over each small interval will be of the same form as in (\ref{infSolnX}), we conclude that $\vnorm{x(t)}_1 \leq \vnorm{x(0)}_1$, for all $t \in [0,T]$. We can therefore write
\begin{eqnarray*}
|p_i(t)-p_j(t)| \leq 2\delta\left(\vnorm{x(0)}_1+n|x_{avg}|\right), \quad t \in [s,T].
\end{eqnarray*}
When we propagate the solution backward over the interval $[T-2\delta,\delta]$, a new integral term will be added to the expression in (\ref{infSolnP}); however, it will be of the same structure as that of $p(t)$ over $[s,T]$: an integral of a doubly stochastic matrix that depends on the system matrix over $[T-2\delta,T-\delta]$ multiplied by $x(t)-\bar{x}$. This applies for all infinitesimal intervals we consider from $T-2\delta$ down to zero. Hence, if we subdivide $[0,T]$ into $K$ intervals of length $\delta$ and apply the triangular inequality to $|p_i(t)-p_j(t)|$, we will obtain
\begin{equation*}
|p_i(t)-p_j(t)| \leq 2T\left(\vnorm{x(0)}_1+n|x_{avg}|\right), \quad t \in [0,T].
\end{equation*}
Finally, recall that $\vnorm{x(0)}_1 \leq n\vnorm{x(0)}_{\infty}$. We therefore have
\[
|f_{ij}| = |p_i-p_j||x_i-x_j| \leq 4nT\left(\vnorm{x(0)}_\infty+|x_{avg}|\right)\vnorm{x(0)}_{\infty},
\]
as claimed. 
\end{proof}

\begin{lemma} \label{lemma::lowerBound}
Given $\delta$, $\epsilon > 0$, one can select the problem horizon $T$ small enough such that $|f_{ij}|\geq\epsilon^2$ for all $t\in [0,T]$.
\end{lemma}
\begin{proof}
We now show that $|f_{ij}(t)|$ can be bounded below uniformly by a positive number. In order to do so, we need to ensure that $|x_i(t)-x_j(t)|$ does not approach zero as $t\to T$. We are seeking a time $t^\star$ such that for a given $\epsilon>0$, we have $|x_i(t)-x_j(t)| \geq \epsilon$ for all $t < t^\star$ and all $i,j$. We can then fix $T<t^\star$ to ensure the existence of a uniform lower bound on $|x_i(t)-x_j(t)|$. Let us again restrict our attention to a small interval $[t_0,t_0+\delta]$ where the system is time-invariant, and let the system matrix over this interval be $A$. We require that the system did not reach equilibrium over this interval, i.e., $x(t_0+\delta) \neq \bar{x}$. Without loss of generality, we assume that $x_1(t_0) > \hdots > x_n(t_0)$\footnote{We are making the implicit assumption that $x_1(0) > \hdots > x_n(0)$.}. Define the following dynamics 
\begin{eqnarray*}
\frac{d}{dt}(\overline{y}_i - x_1(t_0))  & = & \sum_{j\neq i}A_{ij}(x_1(t_0)-\overline{y}_i), \\
\frac{d}{dt}(\underline{y}_i-x_n(t_0)) & = & \sum_{j\neq i}A_{ij}(x_n(t_0)-\underline{y}_i),
\end{eqnarray*}
with initial conditions $\overline{y}_i(t_0) = 2x_1(t_0)$, $\underline{y}_i(t_0) = 2x_n(t_0)$. Note that $\dot{x}_i = \sum_{j\neq i}A_{ij}(x_j-x_i)$. It follows that $\dot{\underline{y}}_i \leq \dot{x}_i \leq \dot{\overline{y}}_i$. By the comparison principle, we conclude that $\underline{y}_i - x_n(t_0) \leq x_i \leq \overline{y}_i-x_1(t_0)$, for $1\leq i \leq n$. Note that we can readily find the solution trajectories for $\overline{y}$ and $\underline{y}$. By defining $a_i = \sum_{j\neq i} A_{ij}$, we can then write
\begin{eqnarray*}
\overline{y}_i - x_1(t_0) & = & e^{-a_i(t-t_0)}x_1(t_0)  ,\\
\underline{y}_i - x_n(t_0) & = & e^{-a_i(t-t_0)}x_n(t_0).
\end{eqnarray*} 
By solving the equation $\overline{y}_{i-i}-x_1(t_0)=\underline{y}_i-x_n(t_0)$, we can find a time $t^\star_i$ when $x_{i-1}$ can potentially meet $x_i$: 
\[
t^\star = \frac{1}{a_{i-1}-a_i}\ln\left(\frac{x_1(t_0)}{x_n(t_0)}\right) + t_0.
\]
If $t^\star_i > t_0 + \delta$, for all $i$, then we need to propagate the solution forward, and keeping in mind that the system matrix could change, until we find a time $t^\star_i$ in some interval $[\tilde{t},\tilde{t}+\delta]$ where $\overline{y}_{i-i}=\underline{y}_i$ for some $i$. Then, for a given $\epsilon>0$, we can select $T < t^\star$ such that $|x_i-x_{i-1}|\geq |\underline{y}_i-\overline{y}_{i-1}| \geq \epsilon$; hence, we conclude that for this choice of $T$ we can guarantee that $|x_i-x_j|\geq \epsilon>0$ for all $i,j$. 

Since $p(T) = 0$, then $|p_i-p_j|\to 0$ as $t\to T$. Given $\epsilon > 0$, there is a time $\tilde{T}_{ij}(\epsilon) < T$ such that $|p_i-p_j|<\epsilon$, for all $t\geq \tilde{T}_{ij}(\epsilon)$. Define the time $\tilde{T} = \min \limits_{1\leq i,j \leq n} \tilde{T}_{ij}(\epsilon)$. Then $|p_i-p_j|\geq \epsilon$ for $t\geq \tilde{T}(\epsilon)$, for all $i,j$. We therefore have $|f_{ij}|=|p_i-p_j||x_i-x_j|\geq \epsilon^2 > 0$.
\end{proof}

\end{document}